\documentclass[11pt,reqno]{amsart}

\usepackage[T1]{fontenc}
\usepackage[utf8]{inputenc}
\usepackage{amsmath,amssymb,amsfonts,amsthm,mathtools}
\usepackage[margin=1.15in]{geometry}
\usepackage[colorlinks=true,linkcolor=blue,citecolor=blue,urlcolor=blue]{hyperref}

\numberwithin{equation}{section}

\newtheorem{theorem}{Theorem}[section]
\newtheorem{proposition}[theorem]{Proposition}
\newtheorem{conjecture}[theorem]{Conjecture}
\theoremstyle{remark}
\newtheorem{remark}[theorem]{Remark}

\DeclareMathOperator{\tr}{tr}
\newcommand{\R}{\mathbb{R}}
\newcommand{\Id}{\mathrm{Id}}
\newcommand{\ii}{\mathrm{i}}
\newcommand{\e}{\mathrm{e}}
\newcommand{\dd}{\,\mathrm{d}}
\newcommand{\norm}[1]{\left\|#1\right\|}

\newcommand{\abs}[1]{\left|#1\right|}

\title[Counterexamples to a velocity-only kinetic Schauder estimate]{Several counterexamples to kinetic Schauder estimates}

\author{Hongjie Dong}
\address{Division of Applied Mathematics, Brown University, Rhode Island, USA, 02912}
\email{hongjie\_dong@brown.edu}

\author{Weinan Wang}
\address{Department of Mathematics, University of Oklahoma, Norman, USA, 73072}
\email{ww@ou.edu}

\thanks{H. Dong was partially supported by the NSF under agreement DMS-2350129. W. Wang was partially supported by the Simons Foundation TSM grant (No. 0007730).}

\date{\today}

\begin{document}

\begin{abstract}
In this paper, we construct counterexamples to the
Schauder estimate for kinetic Fokker--Planck equations, disproving
\cite[Conjecture~1.3]{HW}.  The conjectured estimate would control second
velocity derivatives using only velocity H\"older regularity of the
coefficients and forcing, with no spatial regularity.  We also give a zero-forcing variant, where the
same {conditions are imposed on} a bounded zeroth-order coefficient.  Finally,
in dimensions $d\geq 2$, we give endpoint examples showing that bounded
velocity-independent forcing may produce bounded distributional stationary
solutions whose second velocity derivatives are not locally bounded. 
\end{abstract}

\maketitle

\section{Introduction}

In this paper, we study the regularity of kinetic Fokker--Planck equations
\begin{equation}\label{eq:kfp-intro}
   (\partial_t+v\cdot\nabla_x)f=\tr\!\left(a\,D_v^2 f\right)+c\,f+g,
   \qquad (t,x,v)\in\R\times\R^d\times\R^d,
\end{equation}
the linear equations that underlie the regularity theory of the Landau and Boltzmann equations of kinetic theory \cite{Villani02}.  The operator in \eqref{eq:kfp-intro} is degenerate---elliptic in $v$ only---but the transport field $\partial_t+v\cdot\nabla_x$ transfers velocity regularity into the space and time variables, rendering it hypoelliptic.  This was discovered by Kolmogorov \cite{Kolmogorov}, who computed the fundamental solution in the constant-coefficient case, and explained in general by H\"ormander \cite{Hormander}; see \cite{Bouchut} for the transfer of regularity.  The attendant quantitative theory is organized by the kinetic dilations
\[
(t,x,v)\mapsto(r^2t,\,r^3x,\,rv),\qquad r>0,
\]
under which H\"older regularity $C_v^\alpha$ in velocity is paired with $C_x^{\alpha/3}$ in space and $C_t^{\alpha/2}$ in time.  Within this scaling, Schauder and De Giorgi--Nash--Moser estimates have been developed by many authors \cite{GIMV,ImbertMouhot,ImbertSilvestreSchauder,BramantiBrandolini,DiFrancescoPolidoro,Manfredini,HaoWuZhang,Zhu}; see \cite{SilvestreOpen} for a survey.

In \cite{HW}, Henderson and the second named author of this present paper proved a kinetic
Schauder estimate requiring only measurability in time, with H\"older
regularity \(C_x^{\alpha/3}C_v^\alpha\) in the remaining variables,
and applied it to weak--strong uniqueness for the spatially
inhomogeneous Landau equation.  That estimate was modeled on the
interior parabolic Schauder estimates of Brandt \cite{Brandt} and
Knerr \cite{Knerr}; see also Lieberman \cite{Lieberman}.  A closely
related kinetic Schauder estimate, with coefficients measurable in
time and H\"older continuous in the intrinsic spatial {and velocity} variables, was
also obtained by Biagi and Bramanti \cite{BiagiBramanti}.  {See also a global kinetic Schauder estimates by the first named author and Yastrzhembskiy in \cite{DongYastrzhembskiySchauder} by using a kernel-free approach, where the time regularity of $D_v^2 f$ was also derived.}
It was then natural to ask whether the spatial regularity, like the temporal
regularity, is dispensable in the kinetic setting: can \(D_v^2 f\) be
controlled in \(C_v^\alpha\) by the velocity H\"older seminorms of the
coefficients and forcing alone, with no spatial seminorm on the
right-hand side?  This question was made precise in \cite[Conjecture~1.3]{HW}.

This question sits at the endpoint of a broader positive theory whose
purpose is to lower the regularity assumptions in kinetic Schauder
estimates.  Classical intrinsic Schauder estimates for Kolmogorov--type
operators, going back to Manfredini, Di Francesco--Polidoro,
Bramanti--Brandolini, and Priola, already use the anisotropic geometry
of the equation \cite{Manfredini,DiFrancescoPolidoro,BramantiBrandolini,Priola}.
More recent sharp or kinetic versions, including the estimates of
Imbert--Mouhot, Imbert--Silvestre, and Chaudru de Raynal--Honor\'e--Menozzi,
are likewise formulated in intrinsic H\"older spaces
\cite{ImbertMouhot,ImbertSilvestreSchauder,ChaudruHonoreMenozzi}.  The
recent frontier has been to weaken specific parts of these hypotheses:
time regularity is removed in Henderson--Wang, Biagi--Bramanti,
Dong--Yastrzhembskiy, and Lucertini--Pagliarani--Pascucci
\cite{HW,BiagiBramanti,DongYastrzhembskiySchauder,LucertiniPagliaraniPascucciRegularity,LucertiniPagliaraniPascucciSchauder},
while the spatial H\"older modulus can be replaced by Dini or
log-Dini moduli in the work of Polidoro--Rebucci--Stroffolini and
Biagi--Bramanti--Stroffolini
\cite{PolidoroRebucciStroffolini,BiagiBramantiStroffolini}.  In a
Sobolev direction, the first named author of this present paper and Yastrzhembskiy developed global \(L^p\)
theories under VMO-type assumptions {with respect to $(x,v)$} in nondivergence and divergence
form \cite{DongYastrzhembskiyLpND,DongYastrzhembskiyLpD}; see also the
fundamental-solution work of Bramanti--Polidoro for time-dependent
measurable coefficients \cite{BramantiPolidoro}.  Thus many works push
toward lower or more intrinsic non-time regularity.  What they do not
remove, in any Schauder estimate controlling \(D_v^2 f\) in H\"older
spaces, is the {assumption} of a spatial modulus in the kinetic variables.
The present note settles the endpoint relevant to \cite[Conjecture~1.3]{HW}:
if that spatial modulus is removed completely and only velocity
H\"older information is kept, the general linear estimate is false.

To recall it, for $r>0$ write
\[
   Q_r=(-r^2,0]\times B_{r^3}\times B_r\subset\R\times\R^d\times\R^d,
\]
and, for a bounded function $h$ on a cylinder $Q$, let
\[
   [h]_{C_v^\alpha(Q)}
   :=\sup_{\substack{(t,x,v),(t,x,v')\in Q\\0<|v-v'|<1/2}}
       \frac{|h(t,x,v)-h(t,x,v')|}{|v-v'|^\alpha}
\]
denote its H\"older seminorm in the velocity variable.  We use the analogous notation
\[
   [h]_{C_x^\beta(Q)}
   :=\sup_{\substack{(t,x,v),(t,x',v)\in Q\\0<|x-x'|<1/2}}
       \frac{|h(t,x,v)-h(t,x',v)|}{|x-x'|^\beta}
\]
for spatial H\"older seminorms.  For matrix-valued functions we use the
Frobenius norm,
\[
   |M|=\left(\sum_{i,j=1}^d |M_{ij}|^2\right)^{1/2}.
\]

\begin{conjecture}[{\cite[Conjecture~1.3]{HW}}]\label{conj:false}
Let $\alpha\in(0,1)$ and $\Lambda>1$.  There is a constant
$C=C(d,\alpha,\Lambda)$ such that every smooth solution $f$ of
\eqref{eq:kfp-intro} in $Q_1$ satisfying
\[
   \Lambda^{-1}\Id\le a\le\Lambda\Id,
   \qquad |c|\le\Lambda,
   \qquad |g|\le\Lambda,
\]
and with \(f,D_v^2f,a,c,g\in C_v^\alpha(Q_1)\), obeys
\begin{equation}\label{eq:false-estimate}
\begin{aligned}
   [D_v^2f]_{C_v^\alpha(Q_{1/2})}
   &\le C\Big(1+[c]_{C_v^\alpha(Q_1)}+[a]_{C_v^\alpha(Q_1)}^{1+2/\alpha}\Big)
        \norm{f}_{L^\infty(Q_1)} \\
   &\quad +C\Big(1+[a]_{C_v^\alpha(Q_1)}\Big)[g]_{C_v^\alpha(Q_1)}.
\end{aligned}
\end{equation}
\end{conjecture}
In the nonlinear problems that motivate the theory---regularity and uniqueness for the Landau and non-cutoff Boltzmann equations \cite{ImbertSilvestreBoltzmann,Silvestre17,Fournier}---the coefficients are determined by the solution, and controlling their \emph{spatial} regularity is typically the delicate step; a velocity-only estimate would remove that burden.  Having already dispensed with time regularity entirely, one might hope the spatial regularity is similarly inessential.

We show that it is not.  Conjecture~\ref{conj:false} is false.  Here and below, $N$ is a large positive spatial frequency, which will be sent to infinity.  For every $\alpha\in(0,1)$ and every $d\ge1$ we construct (Theorem~\ref{thm:main}) smooth, time-independent solutions $f_N$ of \eqref{eq:kfp-intro} on the unit cylinder, with $a\equiv\Id$, $c\equiv0$, and a bounded forcing $g_N$ independent of $v$, for which
\[
   \norm{f_N}_{L^\infty(Q_1)}\le C_0\,N^{-2/3},
   \qquad [g_N]_{C_v^\alpha(Q_1)}=0,
   \qquad [D_v^2 f_N]_{C_v^\alpha(Q_{1/2})}\ge c_0\,N^{\alpha/3}.
\]
Every velocity H\"older seminorm on the right-hand side of \eqref{eq:false-estimate} vanishes for these data, so the right-hand side reduces to a multiple of $\norm{f_N}_{L^\infty}\to0$ while the left-hand side diverges.

The construction is elementary.  For this frequency $N$, in one active pair of variables we take
\[
   f_N(x,v)=\operatorname{Re}\!\left(\e^{\ii Nx_1}I_N(v_1)\right),
   \qquad
   I_N(v)=\int_0^\infty\exp\!\left(-\frac{N^2 s^3}{3}\right)\e^{-\ii N s v_1}\dd s,
\]
with forcing $g_N=\cos(Nx)$; the identity $I_N''(v)-\ii N v\,I_N(v)=-1$, obtained by differentiating the integrand in $s$, makes $f_N$ an exact solution.  The forcing is independent of $v$ and hence invisible to $C_v^\alpha$.  But the rescaling $I_N(v)=N^{-2/3}J(N^{1/3}v)$, with $J(y)=\int_0^\infty\e^{-r^3/3}\e^{-\ii ry}\dd r$, gives $I_N''(v)=J''(N^{1/3}v)$, a profile that varies by an order-one amount across a velocity distance $N^{-1/3}$; this forces $[D_v^2 f_N]_{C_v^\alpha}\gtrsim N^{\alpha/3}$.

The example also identifies the missing quantity.  Although
\(g_N(x,v)=\cos(Nx_1)\) is invisible to the velocity seminorm,
\[
   [g_N]_{C_x^{\alpha/3}(Q_1)}\simeq N^{\alpha/3},
\]
which is exactly the size of the lower bound for
\([D_v^2f_N]_{C_v^\alpha(Q_{1/2})}\).  Thus the construction is
consistent with the \(C_x^{\alpha/3}C_v^\alpha\) Schauder estimate of
\cite{HW,DongYastrzhembskiySchauder}, but shows that the spatial component of the anisotropic
kinetic modulus cannot be omitted in general.

The first construction is tailored to the precise H\"older seminorm in
Conjecture~\ref{conj:false}.  We also include variants that clarify the
nature of the obstruction.  Section~\ref{sec:zero-forcing} places the
oscillation in the zeroth-order coefficient and sets \(g\equiv0\).  In
Section~\ref{sec:Linfty-endpoint}, we show that in dimensions \(d\ge2\) the
endpoint failure is even more basic: bounded velocity-independent forcing for
the constant-coefficient Kolmogorov equation may produce bounded 
{strong} stationary solutions without bounded second velocity derivatives.  We first give a smooth
approximating family with uniformly bounded forcing and solutions but
unbounded velocity Hessian norms.  We then give two actual endpoint examples:
one sign-type example, visible from the second-order Riesz kernel, and one
potential-theoretic example obtained from a compactly supported potential with
bounded {and continuous} Laplacian but unbounded mixed Hessian.  These endpoint constructions
reinforce the same message: without a spatial modulus, bounded data are not sufficient to
control \(D_v^2 f\).

All examples are linear, with 
{given} coefficients, and do not
affect the main results of \cite{HW}.  In particular, they are not solutions
of the Landau equation, and they do not settle the corresponding weak--strong
uniqueness question, \cite[Conjecture 1.4]{HW}.  They do, however, rule out
the proposed route through a generic velocity-only linear Schauder estimate.

\section{The counterexamples}

\begin{theorem}\label{thm:main}
Let \(d\ge1\) and \(\alpha\in(0,1)\), and set
\[
   C_0:=\int_0^\infty e^{-r^3/3}\,dr=3^{-2/3}\Gamma(1/3).
\]
There exist constants \(c_0>0\), a number \(N_0\ge1\), and, for every
frequency parameter \(N\ge N_0\), a smooth real-valued stationary solution \(f_N\) of
\begin{equation}\label{eq:kfp-special}
   (\partial_t+v\cdot\nabla_x)f_N=\Delta_v f_N+g_N
   \qquad\text{in }Q_1,
\end{equation}
with \(\|g_N\|_{L^\infty(Q_1)}\le1\) and
\([g_N]_{C_v^\alpha(Q_1)}=0\), such that
\begin{equation}\label{eq:main-bounds}
   \|f_N\|_{L^\infty(Q_1)}\le C_0N^{-2/3},
   \qquad
   [D_v^2f_N]_{C_v^\alpha(Q_{1/2})}\ge c_0N^{\alpha/3}.
\end{equation}
Consequently the estimate \eqref{eq:false-estimate} is false.
\end{theorem}

\begin{proof}
It is enough to use one active spatial and velocity direction.  Thus all functions below depend only on $(x_1,v_1)$; the remaining variables are passive.  Fix a frequency parameter $N\ge1$ and define
\begin{equation}\label{eq:IN-def}
   I_N(\eta)=\int_0^\infty
      \exp\left(-\frac{N^2s^3}{3}\right)\e^{-\ii Ns\eta}\dd s,
\end{equation}
and set
\begin{equation}\label{eq:fN-gN-def}
   F_N(x,v)=\e^{\ii Nx_1}I_N(v_1),
   \qquad
   f_N(x,v)=\operatorname{Re}F_N(x,v),
   \qquad
   g_N(x,v)=\cos(Nx_1).
\end{equation}
These functions are independent of $t$.

We first verify the equation.  Let
\[
   E_N(s,\eta)=\exp\left(-\frac{N^2s^3}{3}\right)\e^{-\ii Ns\eta}.
\]
Then
\[
   \partial_sE_N(s,\eta)=(-N^2s^2-\ii N\eta)E_N(s,\eta).
\]
The exponential decay in $s$ justifies the following differentiations under the integral sign.  Differentiating \eqref{eq:IN-def} twice in $\eta$ gives
\[
\begin{aligned}
   I_N''(\eta)-\ii N\eta I_N(\eta)
   &=\int_0^\infty (-N^2s^2-\ii N\eta)
       \exp\left(-\frac{N^2s^3}{3}\right)\e^{-\ii Ns\eta}\dd s \\
   &=\int_0^\infty \partial_sE_N(s,\eta)\dd s
     =\lim_{s\to\infty}E_N(s,\eta)-E_N(0,\eta)=-1.
\end{aligned}
\]
Since $F_N$ depends only on $(x_1,v_1)$,
\[
   v\cdot\nabla_xF_N=\ii Nv_1\e^{\ii Nx_1}I_N(v_1),
   \qquad
   \Delta_vF_N=\e^{\ii Nx_1}I_N''(v_1).
\]
Therefore
\[
\begin{aligned}
   v\cdot\nabla_xF_N-\Delta_vF_N
   &=\e^{\ii Nx_1}\big(\ii Nv_1I_N(v_1)-I_N''(v_1)\big)
     =\e^{\ii Nx_1}.
\end{aligned}
\]
Taking real parts yields
\[
   (\partial_t+v\cdot\nabla_x)f_N=\Delta_vf_N+g_N,
\]
which proves \eqref{eq:kfp-special}.  Moreover $g_N$ is independent of $v$, hence
\[
   [g_N]_{C_v^\alpha(Q_1)}=0,
   \qquad
   \norm{g_N}_{L^\infty(Q_1)}\le1.
\]

Next, we estimate the size of $f_N$.  From \eqref{eq:IN-def},
\[
   |I_N(\eta)|\le \int_0^\infty \exp\left(-\frac{N^2s^3}{3}\right)\dd s.
\]
With the change of variables $r=N^{2/3}s$, this becomes
\[
   |I_N(\eta)|\le N^{-2/3}\int_0^\infty \e^{-r^3/3}\dd r
   = C_0N^{-2/3},
\]
where
\[
   C_0=\int_0^\infty e^{-r^3/3}\,dr=3^{-2/3}\Gamma(1/3).
\]
Thus
\begin{equation}\label{eq:fN-Linfty}
   \norm{f_N}_{L^\infty(Q_1)}\le C_0N^{-2/3}.
\end{equation}

It remains to prove the lower bound for $[D_v^2f_N]_{C_v^\alpha}$.  
Define
\begin{equation}\label{eq:J-def}
   J(y)=\int_0^\infty \e^{-r^3/3}\e^{-\ii ry}\dd r.
\end{equation}
The same change of variables $r=N^{2/3}s$ gives
\begin{equation}\label{eq:I-J-scaling}
   I_N(\eta)=N^{-2/3}J(N^{1/3}\eta).
\end{equation}
Consequently
\begin{equation}\label{eq:I-second-J}
   I_N''(\eta)=J''(N^{1/3}\eta).
\end{equation}
Furthermore,
\begin{equation}\label{eq:J-second}
   J''(y)=-\int_0^\infty r^2\e^{-r^3/3}\e^{-\ii ry}\dd r.
\end{equation}
In particular,
\begin{equation}\label{eq:J-zero}
   J''(0)=-\int_0^\infty r^2\e^{-r^3/3}\dd r=-1,
\end{equation}
where the last equality follows from $u=r^3/3$.  Since $r^2\e^{-r^3/3}\in L^1(0,\infty)$, the Riemann--Lebesgue lemma gives $J''(Y)\to0$ as $Y\to\infty$.  Choose a fixed $Y>0$ such that
\begin{equation}\label{eq:Y-choice}
   \abs{\operatorname{Re}J''(Y)-\operatorname{Re}J''(0)}\ge \frac12.
\end{equation}
Set $N_0:=\lfloor 8Y^3\rfloor+1$, so that $YN^{-1/3}<1/2$ for every $N\ge N_0$; then both $0$ and $YN^{-1/3}e_1$ lie in $B_{1/2}$.  Evaluating at $t=0$ and $x=0$, we have by \eqref{eq:I-second-J}
\[
   \partial_{v_1v_1}f_N(0,0,v)=\operatorname{Re}J''(N^{1/3}v_1).
\]
Therefore
\[
\begin{aligned}
   [D_v^2f_N]_{C_v^\alpha(Q_{1/2})}
   &\ge
   \frac{\abs{\partial_{v_1v_1}f_N(0,0,YN^{-1/3}e_1)
          -\partial_{v_1v_1}f_N(0,0,0)}}{|YN^{-1/3}e_1|^\alpha} \\
   &\ge \frac{1/2}{Y^\alpha N^{-\alpha/3}}
    =c_0N^{\alpha/3},
\end{aligned}
\]
where $c_0=(2Y^\alpha)^{-1}$.  This proves \eqref{eq:main-bounds}.

Finally, suppose that the velocity-only Schauder estimate \eqref{eq:false-estimate} held.  Fix any $\Lambda>1$.  The functions above satisfy the hypotheses with this $\Lambda$, since $a\equiv\Id$, $c\equiv0$, and $\|g_N\|_{L^\infty(Q_1)}\le1<\Lambda$.  Applying \eqref{eq:false-estimate} gives
\[
   [D_v^2f_N]_{C_v^\alpha(Q_{1/2})}
   \le C(d,\alpha,\Lambda)\norm{f_N}_{L^\infty(Q_1)},
\]
because
\[
   [a]_{C_v^\alpha(Q_1)}=[c]_{C_v^\alpha(Q_1)}=[g_N]_{C_v^\alpha(Q_1)}=0.
\]
Using \eqref{eq:fN-Linfty}, this would imply
\[
   [D_v^2f_N]_{C_v^\alpha(Q_{1/2})}\le C(d,\alpha,\Lambda)C_0N^{-2/3}.
\]
Together with the lower bound above, we get
\[
   c_0N^{\alpha/3}\le C(d,\alpha,\Lambda)C_0N^{-2/3},
\]
which is impossible as $N\to\infty$.  Thus \eqref{eq:false-estimate} is false.
\end{proof}

\begin{remark}[A bounded forcing term does not fix the problem]\label{rem:g-Linfty}
The same example also rules out any uniform estimate obtained from
\eqref{eq:false-estimate} by adding a term depending only on
\(\|g\|_{L^\infty}\), while still omitting all \(x\)-regularity of
\(g\).  Indeed, \(\|g_N\|_{L^\infty(Q_1)}\le1\), whereas
\([D_v^2f_N]_{C_v^\alpha(Q_{1/2})}\to\infty\).
\end{remark}

\section{A zero-forcing variant}\label{sec:zero-forcing}

The preceding construction places the hidden spatial oscillation in the forcing term.  It can instead be placed in the zeroth-order coefficient.

\begin{proposition}\label{prop:zero-forcing}
Let \(d\ge1\) and \(\alpha\in(0,1)\).  There exist constants
\(C,c>0\), a number \(N_0\ge1\), and, for every frequency parameter
\(N\ge N_0\), smooth real-valued stationary solutions \(u_N\) of
\begin{equation}\label{eq:zero-equation}
   (\partial_t+v\cdot\nabla_x)u_N=\Delta_vu_N+c_Nu_N
   \qquad\text{in }Q_1,
\end{equation}
with zero forcing, such that
\[
   \frac12\le u_N\le\frac32,
   \qquad
   \|c_N\|_{L^\infty(Q_1)}\le C,
   \qquad
   [c_N]_{C_v^\alpha(Q_1)}\le C N^{(\alpha-2)/3},
\]
while
\[
   [D_v^2u_N]_{C_v^\alpha(Q_{1/2})}\ge cN^{\alpha/3}.
\]
In particular, the exponent \((\alpha-2)/3\) is negative, so the
velocity H\"older seminorm of \(c_N\) tends to zero as \(N\to\infty\),
while the velocity H\"older seminorm of \(D_v^2u_N\) diverges.  Thus
the velocity-only estimate fails even when \(g\equiv0\).
\end{proposition}

\begin{proof}
Let $f_N$ and $g_N$ be the functions constructed in Theorem~\ref{thm:main}.  For $N$ sufficiently large, \eqref{eq:fN-Linfty} gives $\norm{f_N}_{L^\infty}\le1/2$.  Define
\[
   u_N=1+f_N,
   \qquad
   c_N=\frac{g_N}{1+f_N}.
\]
Then $1/2\le u_N\le3/2$ and $\norm{c_N}_{L^\infty}\le2$.  Since
\[
   (\partial_t+v\cdot\nabla_x)f_N=\Delta_vf_N+g_N,
\]
we have
\[
   (\partial_t+v\cdot\nabla_x)u_N=\Delta_vu_N+c_Nu_N.
\]

It remains to estimate the velocity H\"older seminorm of \(c_N\).
Because \(g_N\) is independent of \(v\), for fixed \((t,x)\) and
\(v,v'\in B_1\),
\[
\begin{aligned}
   |c_N(t,x,v)-c_N(t,x,v')|
   &= |g_N(x)|
      \left|
      \frac1{1+f_N(t,x,v)}
      -
      \frac1{1+f_N(t,x,v')}
      \right|  \\
   &\le 4|f_N(t,x,v)-f_N(t,x,v')|.
\end{aligned}
\]
Hence
\begin{equation}\label{eq:c-holder-f-holder}
   [c_N]_{C_v^\alpha(Q_1)}
   \le 4[f_N]_{C_v^\alpha(Q_1)}.
\end{equation}
Using the scaling
\[
   I_N(\eta)=N^{-2/3}J(N^{1/3}\eta),
\]
and the boundedness of \(J\) and \(J'\), we obtain
\[
   [I_N]_{C^\alpha(\mathbb R)}
   \le N^{-2/3}N^{\alpha/3}[J]_{C^\alpha(\mathbb R)}
   \le C_\alpha N^{(\alpha-2)/3}.
\]
Therefore
\[
   [f_N]_{C_v^\alpha(Q_1)}
   \le C_\alpha N^{(\alpha-2)/3},
\]
and \eqref{eq:c-holder-f-holder} gives
\[
   [c_N]_{C_v^\alpha(Q_1)}
   \le C_\alpha N^{(\alpha-2)/3}.
\]
Since \((\alpha-2)/3<0\), this seminorm tends to zero as \(N\to\infty\).

Finally, $D_v^2u_N=D_v^2f_N$, so Theorem~\ref{thm:main} gives
\[
   [D_v^2u_N]_{C_v^\alpha(Q_{1/2})}\ge cN^{\alpha/3}.
\]
If the velocity-only estimate \eqref{eq:false-estimate} were true with $g\equiv0$, choose a fixed $\Lambda>1$ large enough that $|c_N|\le\Lambda$ for all large $N$.  Then its right-hand side would be uniformly bounded for this sequence, because $\|u_N\|_{L^\infty}$ is bounded and $[c_N]_{C_v^\alpha(Q_1)}\to0$, while its left-hand side diverges.  This contradiction proves the proposition.
\end{proof}

\section{A bounded-forcing \texorpdfstring{$L^\infty$}{L-infinity} endpoint obstruction}\label{sec:Linfty-endpoint}

The preceding examples disprove the precise velocity H\"older estimate in
Conjecture~\ref{conj:false}.  In dimensions \(d\ge2\), one can see an even
more elementary endpoint obstruction: bounded {(and even continuous)} forcing, independent of
velocity, does not in general produce bounded second velocity derivatives.
The mechanism is the classical failure of second-order Riesz transforms to
map \(L^\infty\) to \(L^\infty\).

We state the result first in a smooth approximate form.  This keeps the
solutions classical while showing that {there is} no uniform \(L^\infty\) estimate for
\(D_v^2f\) {which depends} only on \(\|g\|_{L^\infty}\) and \(\|f\|_{L^\infty}\),
when no spatial modulus of \(g\) is assumed.  We then record actual bounded
forcing examples whose velocity Hessians are not locally bounded.

\begin{theorem}\label{prop:Linfty-endpoint-smooth}
Let \(d\ge2\).  There exist constants \(C,c>0\), a family
\(g_\varepsilon{=g_\varepsilon(x)}\in C_c^\infty(\R^d)\), \(0<\varepsilon<1/10\), and smooth
stationary solutions \(f_\varepsilon\) of
\begin{equation}\label{eq:Linfty-endpoint-equation}
   (\partial_t+v\cdot\nabla_x)f_\varepsilon
   =\Delta_v f_\varepsilon+g_\varepsilon
   \qquad\text{in }\R_t\times\R_x^d\times\R_v^d,
\end{equation}
where \(g_\varepsilon\) is independent of \(v\), such that
\[
   \|g_\varepsilon\|_{L^\infty(\R^d)}\le1,
   \qquad
   \|f_\varepsilon\|_{L^\infty(\R^d\times\R^d)}\le C,
\]
but
\[
   \|D_v^2f_\varepsilon\|_{L^\infty(B_{1/4}\times B_{1/4})}
   \ge c\log\frac1\varepsilon.
\]
\end{theorem}

\begin{proof}
Choose \(\chi\in C_c^\infty(B_1)\), \(0\le\chi\le1\), with
\(\chi\equiv1\) on \(B_{1/2}\).  Choose an odd function
\(\sigma\in C^\infty(\R)\) such that \(|\sigma|\le1\),
\(s\sigma(s)\ge0\), and \(\sigma(s)=\operatorname{sgn}(s)\) for
\(|s|\ge1\).  For \(0<\varepsilon<1/10\), set
\begin{equation}\label{eq:g-eps-def}
   g_\varepsilon(x)
   :=\chi(x)\sigma(x_1/\varepsilon)\sigma(x_2/\varepsilon).
\end{equation}
Then \(g_\varepsilon\in C_c^\infty(\R^d)\), \(\|g_\varepsilon\|_{L^\infty}\le1\),
and \(g_\varepsilon\) is independent of \(v\).

Let \(P_\tau=e^{\tau\Delta_x}\) be the heat semigroup in the \(x\)-variable,
and define
\begin{equation}\label{eq:f-eps-resolvent}
   f_\varepsilon(x,v)
   :=\int_0^\infty
      \big(P_{s^3/3}g_\varepsilon\big)(x-sv)\,\dd s.
\end{equation}
This integral is uniformly bounded in \(\varepsilon\).  Indeed,
\[
   \|P_{s^3/3}g_\varepsilon\|_{L^\infty}
   \le
   \min\big\{\|g_\varepsilon\|_{L^\infty},
          C_d s^{-3d/2}\|g_\varepsilon\|_{L^1}\big\},
\]
and \(\|g_\varepsilon\|_{L^1}\le |B_1|\).  Since
\(\int_0^\infty \min\{1,s^{-3d/2}\}\,\dd s<\infty\),
\[
   \|f_\varepsilon\|_{L^\infty(\R^d\times\R^d)}\le C.
\]
The same heat-kernel bounds justify differentiating under the integral sign.

We next verify the equation.  Set
\[
   G_s(x,v):=\big(P_{s^3/3}g_\varepsilon\big)(x-sv).
\]
Then
\[
   \partial_sG_s
   =s^2\Delta_x P_{s^3/3}g_\varepsilon(x-sv)
    -v\cdot\nabla_x P_{s^3/3}g_\varepsilon(x-sv),
\]
while
\[
   \Delta_vG_s
   =s^2\Delta_x P_{s^3/3}g_\varepsilon(x-sv).
\]
Hence
\[
   \partial_sG_s=\Delta_vG_s-v\cdot\nabla_xG_s.
\]
Integrating in \(s\), using \(G_s\to0\) as \(s\to\infty\) and
\(G_0=g_\varepsilon\), gives
\[
   \Delta_vf_\varepsilon-v\cdot\nabla_xf_\varepsilon=-g_\varepsilon,
\]
which is equivalent to \eqref{eq:Linfty-endpoint-equation}. It remains to prove the lower bound for the velocity Hessian.  At \(v=0\),
\[
\begin{aligned}
   \partial_{v_1v_2}f_\varepsilon(x,0)
   &=\int_0^\infty
      s^2\partial_{x_1x_2}P_{s^3/3}g_\varepsilon(x)\,\dd s  \\
   &=\int_0^\infty
      \partial_{x_1x_2}P_\tau g_\varepsilon(x)\,\dd\tau,
\end{aligned}
\]
where \(\tau=s^3/3\).  Therefore
\begin{equation}\label{eq:riesz-identity}
   \partial_{v_1v_2}f_\varepsilon(\cdot,0)
   =\partial_{x_1x_2}(-\Delta_x)^{-1}g_\varepsilon,
\end{equation}
where \(\partial_{x_1x_2}(-\Delta_x)^{-1}\) is understood as the
second-order Riesz transform.  For the mixed derivative, the kernel of
\(\partial_{x_1x_2}(-\Delta_x)^{-1}\) is
\begin{equation}\label{eq:riesz-kernel}
   K_{12}(y)=c_d\frac{y_1y_2}{|y|^{d+2}},
   \qquad y\ne0,
\end{equation}
with \(c_d>0\).  For the particular function \(g_\varepsilon\) in
\eqref{eq:g-eps-def}, the product \(K_{12}g_\varepsilon\) is locally
integrable at the origin.  Hence the value at the origin is finite and
can be computed from this kernel:
\[
   \partial_{v_1v_2}f_\varepsilon(0,0)
   =\int_{\R^d}K_{12}(y)g_\varepsilon(y)\,\dd y.
\]
The integrand is nonnegative, because
\(y_1y_2\sigma(y_1/\varepsilon)\sigma(y_2/\varepsilon)\ge0\).  On the fixed
angular cone
\[
   \Omega:=\{\omega\in S^{d-1}: |\omega_1|\ge1/4,\ |\omega_2|\ge1/4\},
\]
which has positive surface measure, and for
\(4\varepsilon<r<1/4\), we have \(\chi(r\omega)=1\) and
\(\sigma(r\omega_j/\varepsilon)=\operatorname{sgn}(\omega_j)\) for
\(j=1,2\).  Hence
\[
\begin{aligned}
   \partial_{v_1v_2}f_\varepsilon(0,0)
   &\ge
   c\int_{4\varepsilon}^{1/4}\int_{\Omega}
      \frac{|r\omega_1\,r\omega_2|}{r^{d+2}}
      r^{d-1}\,\dd\omega\,\dd r  \\
   &\ge c\int_{4\varepsilon}^{1/4}\frac{\dd r}{r}
    \ge c\log\frac1\varepsilon.
\end{aligned}
\]
This proves the claimed lower bound for \(\|D_v^2f_\varepsilon\|_{L^\infty}\).
\end{proof}

\begin{remark}[An unbounded-Hessian example]\label{rem:Linfty-unbounded}
The smooth sequence above can be viewed as an approximation to an actual
bounded-forcing example with unbounded Hessian.  Let
\[
   g(x)=\chi(x)\operatorname{sgn}(x_1x_2),
\]
with the same cutoff \(\chi\), and define
\[
   f(x,v)=\int_0^\infty \big(P_{s^3/3}g\big)(x-sv)\,\dd s.
\]
Then \(g\in L^\infty\cap L^1\), \(g\) is independent of \(v\), and the same
heat-kernel estimate gives \(f\in L^\infty\).  The equation
\[
   (\partial_t+v\cdot\nabla_x)f=\Delta_vf+g
\]
holds in the sense of distributions {and also in the strong sense}.  Moreover,
\[
   \partial_{v_1v_2}f(\cdot,0)
   =\partial_{x_1x_2}(-\Delta_x)^{-1}g
\]
in the sense of distributions.  The kernel representation
\eqref{eq:riesz-kernel} shows that, at the level of truncated singular
integrals, the contribution near the origin is comparable on a cone to
\(|y|^{-d}\), whose radial integral is logarithmically divergent.  Equivalently,
the smooth approximants \(g_\varepsilon\) above converge pointwise to \(g\) away
from the coordinate axes, while the corresponding values
\(\partial_{v_1v_2}f_\varepsilon(0,0)\) grow like \(\log(1/\varepsilon)\).
Thus \(D_v^2f\notin L^\infty_{\mathrm{loc}}\).  This is a stronger endpoint
failure than the one needed for Conjecture~\ref{conj:false}, but it relies on a
rough bounded forcing and, unlike Theorem~\ref{thm:main}, requires at least two
active spatial directions.

\end{remark}

\begin{theorem}[Another endpoint example]\label{prop:potential-endpoint}
Let \(d\ge2\).  There exist \(g\in {C_c}(\R^d)\), independent of
\(v\), and a bounded distributional stationary solution \(f\) of
\[
   (\partial_t+v\cdot\nabla_x)f=\Delta_v f+g
   \qquad\text{in }\R_t\times\R_x^d\times\R_v^d,
\]
such that \(D_v^2f\notin L^\infty_{\mathrm{loc}}\).  Moreover, there are {universal constants $C,c>0$ and}
smooth compactly supported forcings \(g_\varepsilon\), independent of \(v\),
with \(\|g_\varepsilon\|_{L^\infty}\le C\), and smooth bounded stationary
solutions \(f_\varepsilon\), such that
\[
   \|f_\varepsilon\|_{L^\infty}\le C,
   \qquad
   \|D_v^2f_\varepsilon\|_{L^\infty}\ge c\left(\log\frac1\varepsilon\right)^{1/2}.
\]
\end{theorem}

\begin{proof}
Let \(\eta\in C_c^\infty(B_{1/4})\) satisfy \(0\le\eta\le1\) and
\(\eta\equiv1\) on \(B_{1/8}\).  For \(0<|x|<1/4\), 
define
\begin{equation}\label{eq:potential-u-def}
   u(x):=\eta(x)x_1x_2 (-\log |x|)^{1/2},
\end{equation}
with \(u(0)=0\).  We claim that
\[
   g:=\Delta_x u
\]
is a bounded {continuous} compactly supported function, whereas \(\partial_{x_1x_2}u\) is
unbounded near the origin.

Indeed, near the origin, where \(\eta\equiv1\), write
\[
   r=|x|,
   \qquad
   \phi(r)=(-\log r)^{1/2}.
\]
Thus \(u=x_1x_2\phi\) near the origin.  We first show that
\(g=\Delta_xu\) is bounded {and continuous}.  The polynomial \(x_1x_2\) is harmonic and homogeneous
of degree two.  Hence
\begin{equation}\label{eq:laplacian-Pphi}
   \Delta_x(x_1x_2\phi(r))
   =x_1x_2\left(\phi''(r)+\frac{d+3}{r}\phi'(r)\right)
\end{equation}
is continuous in $B_1\setminus{0}$.
Moreover
\[
   \phi'(r)=-\frac{1}{2r\phi(r)},
   \qquad
   \phi''(r)=\frac{1}{2r^2\phi(r)}
      -\frac{1}{4r^2\phi(r)^3}.
\]
Consequently
\[
   \phi''(r)+\frac{d+3}{r}\phi'(r)
   =-\frac{d+2}{2r^2\phi(r)}
    -\frac{1}{4r^2\phi(r)^3}.
\]
Using \(|x_1x_2|\le r^2/2\), we obtain
\[
   |\Delta_x (x_1x_2\phi(r))|
   \le \frac{C|x_1x_2|}{r^2\phi(r)}
       +\frac{C|x_1x_2|}{r^2\phi(r)^3}
   \le C\quad{\text{in}\,\,B_{1/4}\quad\text{and}\quad
\Delta_x (x_1x_2\phi(r))\to 0\quad\text{as}\,\,x\to 0}.
\]
The cutoff 
{function $\eta$ is supported in $B_{1/4}$}, so
\[
   g=\Delta_xu\in {C_c}(\R^d).
\]

We next show that the mixed {derivative} of \(u\) is unbounded.  For a radial
function \(\phi(r)\),
\[
   \partial_i\phi=\phi'(r)\frac{x_i}{r},
\]
and, for \(i\ne j\),
\[
   \partial_{ij}\phi
   =\left(\phi''(r)-\frac{\phi'(r)}{r}\right)\frac{x_ix_j}{r^2}.
\]
The product rule gives
\begin{align}
   \partial_{x_1x_2}(x_1x_2\phi)
   &=\phi
     +x_2\partial_{x_2}\phi
     +x_1\partial_{x_1}\phi
     +x_1x_2\partial_{x_1x_2}\phi \notag \\
   &=\phi
     +\frac{x_1^2+x_2^2}{r}\phi'(r)
     +\frac{x_1^2x_2^2}{r^2}
       \left(\phi''(r)-\frac{\phi'(r)}{r}\right).
   \label{eq:potential-mixed-expanded}
\end{align}
From the formulas for \(\phi'\) and \(\phi''\),
\[
   \phi''(r)-\frac{\phi'(r)}{r}
   =\frac{1}{r^2\phi(r)}-
     \frac{1}{4r^2\phi(r)^3}.
\]
Using \(x_1^2+x_2^2\le r^2\) and \(x_1^2x_2^2\le r^4\), the last two
terms in \eqref{eq:potential-mixed-expanded} are bounded by \(C\phi(r)^{-1}\).
Therefore
\begin{equation}\label{eq:potential-mixed-unbounded}
   \partial_{x_1x_2}\big(x_1x_2\phi(|x|)\big)
   =\phi(|x|)+O\big(\phi(|x|)^{-1}\big)
   \qquad\text{as }x\to0.
\end{equation}
Consequently
\[
   \partial_{x_1x_2}u(x)
   =(-\log|x|)^{1/2}+O\big((-\log|x|)^{-1/2}\big),
   \qquad x\to0,
\]
and \(\partial_{x_1x_2}u\notin L^\infty_{\mathrm{loc}}\).

Define
\begin{equation}\label{eq:potential-f-def}
   f(x,v):=\int_0^\infty \big(P_{s^3/3}g\big)(x-sv)\,\dd s,
\end{equation}
where \(P_\tau=e^{\tau\Delta_x}\).  As in the proof of
Proposition~\ref{prop:Linfty-endpoint-smooth}, the compact support and boundedness
of \(g\) imply \(f\in L^\infty\), and \(f\) solves
\[
   (\partial_t+v\cdot\nabla_x)f=\Delta_vf+g
\]
{in the strong sense}.  Also, at \(v=0\),
\[
   \partial_{v_1v_2}f(\cdot,0)
   =\partial_{x_1x_2}(-\Delta_x)^{-1}g.
\]
Since \(g=\Delta_xu\), the function \(-u\) solves
\(-\Delta_x(-u)=g\).  Thus \((-\Delta_x)^{-1}g=-u\), up to an additive
harmonic term. {Since $(-\Delta_x)^{-1}g(x)\to 0$ as $|x|\to \infty$, by the maximum principle, we see that \((-\Delta_x)^{-1}g=-u\).} 
Therefore
\[
   \partial_{v_1v_2}f(\cdot,0)=-\partial_{x_1x_2}u,
\]
which is unbounded near the origin.  This proves the first assertion. 

For the smooth version, let \(\rho_\varepsilon=\rho_\varepsilon(x)\) be a standard nonnegative
radial mollifier and set
\[
   u_\varepsilon=\rho_\varepsilon*u,
   \qquad
   g_\varepsilon=\Delta_xu_\varepsilon=\rho_\varepsilon*g.
\]
Then \(g_\varepsilon\in C_c^\infty\), \(\|g_\varepsilon\|_{L^\infty}\le
\|g\|_{L^\infty}\), and the corresponding solutions
\[
   f_\varepsilon(x,v)=\int_0^\infty
      \big(P_{s^3/3}g_\varepsilon\big)(x-sv)\,\dd s
\]
are uniformly bounded.  Moreover,
\[
   \partial_{v_1v_2}f_\varepsilon(\cdot,0)
   =-\partial_{x_1x_2}u_\varepsilon.
\]
Using \eqref{eq:potential-mixed-unbounded} and the positivity of the mollifier,
we get a quantitative lower bound.  Indeed, choose the standard mollifier
\(\rho\) so that \(\rho\ge c_\rho>0\) on \(B_{1/2}\).  By
\eqref{eq:potential-mixed-unbounded}, after decreasing \(\varepsilon\) if
necessary,
\[
   \partial_{x_1x_2}u(-y)
   \ge \frac12\left(\log\frac1{|y|}\right)^{1/2}
   \qquad\text{for }0<|y|<\varepsilon.
\]
Therefore, we see
\[
\begin{aligned}
   |\partial_{x_1x_2}u_\varepsilon(0)|
   &=\left|\int_{\R^d}\rho_\varepsilon(y)
        \partial_{x_1x_2}u(-y)\,\dd y\right| \\
   &\ge c\int_{|y|<\varepsilon/2}\rho_\varepsilon(y)
        \left(\log\frac1{|y|}\right)^{1/2}\,\dd y \\
   &\ge c\left(\log\frac1\varepsilon\right)^{1/2}
\end{aligned}
\]
for all sufficiently small \(\varepsilon\).  Hence, we see
\[
   \|D_v^2f_\varepsilon\|_{L^\infty}
   \ge c\left(\log\frac1\varepsilon\right)^{1/2},
\]
which proves the smooth approximation statement.
\end{proof}

\end{document}